\newcommand\Modified{01-04-2015, 2015 }
\documentclass[11pt]{amsart}

\usepackage{amsfonts, amsmath,amscd, amssymb, latexsym, mathrsfs, mathtools, slashed, stmaryrd, verbatim, wasysym }
\usepackage[all]{xy}
\usepackage{hyperref}

\newcommand\datver[1]{\def\datverp%
{\par\boxed{\boxed{\text{Version: #1; Run: \today}}}}}

\newtheorem{theorem}{Theorem}

\newtheorem{definition}[theorem]{Definition}
\newtheorem{corollary}[theorem]{Corollary}

\newtheorem{lemma}[theorem]{Lemma}

\newtheorem{proposition}[theorem]{Proposition}
\newtheorem{remark}[theorem]{Remark}

\numberwithin{equation}{section}
\numberwithin{theorem}{section}

 \usepackage{color}
\definecolor{darkgreen}{cmyk}{1,0,1,.2}
\definecolor{m}{rgb}{1,0.1,1}
\renewcommand{\tilde}{\widetilde}
\newcommand{\wt}[1]{\widetilde{#1}}

\newcommand\eps\varepsilon

\newcommand\pa{\partial}

\newcommand{\End}{\operatorname{End}}

\newcommand{\loc}{\operatorname{loc}}

\newcommand{\reg}{ \mathrm{reg} }

\newcommand{\supp}{\operatorname{supp}}

\DeclareMathAlphabet{\mathpzc}{OT1}{pzc}{m}{it}

\newcommand\paperintro%
        {%
         }
\newcommand\paperbody%
        {%
         }

\newcommand\bbC{\mathbb{C}}

\newcommand\bbP{\mathbb{P}}

\def\RR{\bbR}
\newcommand\bbR{\mathbb{R}}

\newcommand\cC{\mathcal{C}}
\newcommand\cD{\mathcal{D}}

\newcommand\cI{\mathcal{I}}

\newcommand\cK{\mathcal{K}}

\newcommand\cO{\mathcal{O}}

\newcommand\cU{\mathcal{U}}

\newcommand\cW{\mathcal{W}}

\newcommand\sing{\operatorname{sing}}
\newcommand\wtV{\widetilde{V}}

\datver{\Modified}

\begin{document}

\title{On the $L^2$-$\overline{\partial}$-cohomology of certain complete K\"ahler metrics}

%\author{Pierre Albin}
%\address{University of Illinois at Urbana-Champaign}
%\email{palbin@illinois.edu}
\author{Francesco Bei}
\address{Institut Camille Jordan, Universit\'e de Lyon1}
\email{bei@math.univ-lyon1.fr }
\author{Paolo Piazza}
\address{Dipartimento di Matematica, Sapienza Universit\`a di Roma}
\email{piazza@mat.uniroma1.it}

\begin{abstract}
Let $V$ be a compact and irreducible complex space of complex dimension $v$ whose regular part is endowed with a complete Hermitian metric $h$. Let $\pi:M\rightarrow V$ be a  resolution of $V$. Under suitable assumptions on $h$ we prove that
%Let $M$ be a compact K\"ahler manifold, let $V\subset M$ be a pure-dimensional analytic subvariety of complex dimension $n$ and let $\pi:\wt V\rightarrow V$ be a Hironaka resolution of $V$. We consider on  $\reg(V)$  a  Saper-type complete K\"ahler  metric $g_S$ as defined in  \cite{GMM}. Our main result are the following isomorphisms for the $L^2$-$\overline{\pa}$-cohomology of $(\reg(V),g_S)$: 
$$H^{v,q}_{2,\overline{\pa}}(\reg(V),h)\cong H^{v,q}_{\overline{\pa}}(M),\  q=0,...,v.$$ Then we show that the previous isomorphism applies to the case of  Saper-type  K\"ahler  metrics, see  \cite{GMM}, and to the case of complete K\"ahler metrics with finite volume and pinched negative sectional curvatures.
\end{abstract}

\maketitle

%\vspace{1 cm}

\vspace{0.3 cm}

%\subjclass[2010]{Primary:58J10, 32W05.}

%\setcounter{tocdepth}{1}
\tableofcontents

%%%%%%%%%%%%%%%%%%%%%%%%
\paperintro
\section*{Introduction and main results}
%%%%%%%%%%%%%%%%%%%%%%%%

Let $V$ be a complex  projective variety in $\bbC \bbP^m$ of complex dimension $v$. Let $\reg (V)$ be its regular part
and let $\sing (V)$ be its singular locus. More generally, let $N$ be a compact K\"ahler manifold of complex dimension $n$ with K\"ahler form $\omega$ and let $V$ be an  analytic subvariety  in $N$ of complex dimension $v$. One of the questions raised in \cite{CGM} concerns the existence of a complete K\"ahler metric on $\reg(V)$ whose $L^2$-cohomology is isomorphic to the middle perversity intersection cohomology of $V$. This problem has been investigated by Saper in \cite{LS} who provided an affirmative answer in the setting of isolated singularities. Labeling with $g_S$ the K\"ahler metric constructed by Saper, a problem related  to the previous one, is to understand the $L^2$-$\overline{\partial}$-cohomology of $g_S$.  For the case of $(v,q)$-forms this latter  problem has been addressed again by Saper in the setting of isolated singularities. He showed in \cite{LS} that
\begin{equation}
H^{v,q}_{2,\overline{\partial}}(\reg(V),g_S)\cong H^{v,q}_{\overline{\partial}}(M)
\end{equation}
where $\pi:M\rightarrow V$ is any resolution of $V$.  An analogous question,  but for the $L^2$-$\overline{\pa}$-cohomology of the incomplete K\"ahler metric  $g$ on $\reg(V)$ induced by the  K\"ahler metric on $N$, has been considered by  Pardon \cite{P}, 
Haskell \cite{PH}, Br\"uning-Peyerimhoff-Schr\"oder \cite{BPS} and Pardon-Stern \cite{PS} who finally solved the MacPherson conjecture \cite{Mac} by showing, more generally,  that $$H^{v,q}_{2,\overline{\partial}_{\min}}(\reg(V),g)\cong H^{v,q}_{\overline{\partial}}(M).$$

Subsequently, inspired by the work of Saper \cite{LS}, Grant Melles and Milman have constructed  a new family of complete K\"ahler metrics on $\reg (V)$, see \cite{GMM} and  \cite{GMMI}, without any assumption on the singular set of $V$.  The complete K\"ahler metrics built by Grant Melles and Milman on $\reg (V)$ are called 
{\it Saper-type K\"ahler metrics}.  The purpose  of this paper is to investigate the $L^2$-$\overline{\partial}$-cohomology groups of various complete Hermitian metrics that are defined on the regular part of a compact and irreducible complex space;
Saper-type K\"ahler  metrics are examples of such Hermitian metrics.
The paper is structured in the following way. In the first section we provide a general result in the  setting of compact and irreducible complex space whose regular part is endowed with a complete Hermitian metric. 
Our theorem reads as follow:
\begin{theorem}
\label{intro}
Let $V$ be a compact and irreducible complex space of complex dimension $v$. Let $\pi:M\rightarrow V$ be a  resolution of $V$ with $D:=\pi^{-1}(\sing(V))$ a  normal crossings divisor of $M$. Let $h$ be a complete Hermitian metric on $\reg(V)$ and let $\sigma$ be the complete Hermitian metric on $M\setminus D$ defined as $\sigma:=(\pi|_{M\setminus D})^*h$. Assume that:
\begin{itemize}
\item $\pi_*\cC^{v,q}_{D,\sigma}$ is a fine sheaf for each $q=0,...,v$,
\item for each $p\in \sing(V)$ there is an open neighborhood $U$ of p and a  $d$-bounded K\"ahler metric $g_{U}$ on $\reg(U)$ such that $h|_{U}$ and $g_{U}$ are quasi-isometric. 
\end{itemize}
 Then we have the following isomorphism for each $q=0,...,v$:
$$H^{v,q}_{2,\overline{\partial}}(\reg(V),h)\cong H^{v,q}_{\overline{\partial}}(M).$$
\end{theorem}  

\noindent
The sheaves $\cC^{v,q}_{D,\sigma}$ on $M$ are obtained by sheafification from the
presheaves $C^{v,q}_{D,\sigma}$  that assign to an open set 
$U\subset M$ the maximal domain  of $\overline{\partial}_{v,q,{\rm max}}$ on $(U\setminus U\cap D,
\sigma|_{U\setminus U\cap D})$.
For more on this definition 
and for the notion of $d$-bounded K\"ahler metric appearing in the second condition we refer the reader to  \eqref{sheafifi} and Def. \ref{bound} respectively.  As a consequence of Th. \ref{intro} we obtain  that
$$\overline{\pa}_{v} + \overline{\pa}^*_{v} : L^2 \Omega^{v,\bullet} (\reg(V),h)\to L^2 \Omega^{v,\bullet} (\reg(V),h)$$ is a Fredholm operator on its domain endowed with the graph norm, see \eqref{highland}. Moreover, under suitable assumptions, the groups $H^{v,q}_{2,\overline{\partial}}(\reg(V),h)$ are invariant under bimeromorphic maps.\\The second section of this paper is devoted to some applications of Th. \ref{intro}. Our goal here is to provide some examples of complete K\"ahler metrics that obey the hypothesis of Th. \ref{intro}. As a first application we discuss the case of  Saper-type K\"ahler metrics. Our main result is the following
\begin{theorem}\label{theo:Main}
Let $M$ be a compact K\"ahler manifold with K\"ahler form $\omega$ and let $V$ be
an analytic subvariety  of $M$ with complex dimension $v$. Denote now by  $\pi:\wtV\rightarrow V$  a  resolution of $V$. Finally let $g_S$ be a Saper-type metric  on $\reg (V)$ as in \cite{GMM}. 
Then  the following isomorphism holds:
\begin{equation}\label{main-Iso}
H^{v,q}_{2,\overline{\pa}} (\reg (V), g_S)\cong H^{v,q}_{\overline{\pa}} (\widetilde{V})\,
\end{equation}
for every $q=0,...,v.$
\end{theorem}

The second example that we present are complete K\"ahler manifolds with finite volume and pinched negative sectional curvatures. According to a result of Siu-Yau \cite{SiuYau}, a complex manifold that carries a K\"ahler metric with these properties is  biholomorphic to the regular part of a complex projective variety with only isolated singularities.  In this setting, using our Th. \ref{intro},  we prove the following result:
\begin{theorem}
\label{negativek}
Let $(M,h)$ be a complete K\"ahler manifold of complex dimension $m$ with finite volume. Assume that the sectional curvatures of $(M,h)$ satisfies  $-b^2\leq \sec_h\leq -a^2$ for some constants $0<a\leq b$. Let $V\subset \mathbb{C}\mathbb{P}^n$ be the Siu--Yau compactification of $M$ and let $\pi:\tilde{V}\rightarrow V$ be a resolution of $V$. Then we have the following isomorphism for each $q=0,...,m$ $$H^{m,q}_{2,\overline{\pa}}(M,h)\cong H^{m,q}_{\overline{\pa}}(\tilde{V}).$$ 
\end{theorem}

\noindent
In the second section we also collect various consequences  of  Th. \ref{main-Iso} and Th. \ref{negativek}.
\vspace{1 cm}

\smallskip
\noindent
{\bf Acknowledgements.} 
This work was performed within the framework of the LABEX MILYON (ANR-10-LABX-0070) of Universit\'e de Lyon, within the program "Investissements d'Avenir" (ANR-11-IDEX-0007) operated by the French National Research Agency (ANR). Moreover the first author wishes to thank also {\it SFB 647: Raum-Zeit-Materie} for financial support.  Part of this work was done while the first author was visiting Sapienza Universit\`a di Roma whose hospitality and financial support are gratefully acknowledged. It is a pleasure  to thank Pierre Albin for interesting  discussions. We also wish to thank the referee of a first version of this
paper for very interesting remarks and suggestions.

\paperbody

\section{$d$-bounded K\"ahler forms and $L^2$-$\overline{\partial}$-cohomology}

We start with the following remarks about our notation. Let $(M,h)$ be a complex Hermitian manifold. For any $(p,q)$ the maximal extension of $\overline{\pa}_{p,q}$, labeled by $\overline{\pa}_{p,q,\max}:L^2\Omega^{p,q}(M,h)\rightarrow L^2\Omega^{p,q+1}(M,h)$, is the closed extension defined in the distributional sense: $\omega\in \mathcal{D}(\overline{\pa}_{p,q,\max})$ if $\omega\in L^2\Omega^{p,q}(M,h)$ and $\overline{\pa}_{p,q}\omega$, applied in the distributional sense, lies in $L^2\Omega^{p,q+1}(M,g)$.  The minimal extension of $\overline{\pa}_{p,q}$, labeled by $\overline{\pa}_{p,q,\min}:L^2\Omega^{p,q}(M,h)\rightarrow L^2\Omega^{p,q+1}(M,h)$, is defined as the graph closure of $\Omega^{p,q}_c(M)$ in $L^2\Omega^{p,q}(M,h)$ with respect to graph norm of $\overline{\pa}_{p,q}$. It is easy to check that in both cases we get a complex whose corresponding cohomology is denoted by $H^{p,q}_{2,\overline{\pa},\max/\min}(M,h)$. If $(M,h)$ is complete then it is well known that $\overline{\partial}_{p,q,\min}=\overline{\partial}_{p,q,\min}$ and we label this unique closed extension simply with $\overline{\pa}_{p,q}:L^2\Omega^{p,q}(M,h)\rightarrow L^2\Omega^{p,q+1}(M,h)$. Finally analogous notations and considerations hold for the operator $\overline{\pa}+\overline{\pa}^t$. Now we go on with the following definition.
\begin{definition}
\label{bound}
Let $(M,g)$ be a K\"ahler manifold and let $\omega$ be the corresponding K\"ahler form. We will say that $\omega$ is $d$-bounded if there exists a $1$-form $\eta\in L^{\infty}\Omega^1(M,g)$ such that $d\eta=\omega$ where $d\eta$ is understood in the distributional sense.
\end{definition}

\begin{remark}
The notion of $d$-bounded K\"ahler form has been introduced by Gromov in \cite{MGR}. The definition that we stated above is slightly more general than the original one because we do not require $\eta$ to be a smooth form. 
\end{remark}

\begin{definition}
\label{ohsawa}
Let $(M,g)$ be a K\"ahler manifold and let $\omega$ be the corresponding K\"ahler form. We will say that $g$
satisfies the Ohsawa condition if there exists a function $f\in C^{\infty}(M)$ such that $\frac{i}{2\pi}\partial\overline{\partial}f=\omega$ with $\partial f\in L^{\infty}\Omega^{1,0}(M,g)$. See for example \cite{GMM}.
\end{definition}

\begin{remark}
It is clear that if $g$ satisfies the Ohsawa condition then the corresponding K\"ahler form $\omega$ is $d$-bounded.
\end{remark}

We recall now from \cite{MGR} the following important result.
\begin{theorem}
\label{L2vanishing}
In the setting of definition \ref{bound}. Assume moreover that $(M,g)$ is complete and let $m$ be the complex dimension of $m$. Then  for any $(p,q)$ with $p+q\neq m$ we have $$H^{p,q}_{2,\overline{\partial}}(M,g)=0.$$
\end{theorem}
\begin{proof}
When $\eta$ is smooth this theorem is Th.1.4.A in \cite{MGR}. A careful look at the arguments used there shows that the same proof applies also in our slightly  more general setting. Finally we point out that if $g$ satisfies the Ohsawa condition then the theorem had been already proved in \cite{DF} and \cite{Takeo}.
\end{proof}

Furthermore, concerning $d$-bounded K\"ahler metrics, we have also the following basic properties.

\begin{proposition}
\label{pullback}
Let $f:M\rightarrow N$ be an holomorphic immersion between complex manifolds. Let $g$ be a $d$-bounded K\"ahler metric on $N$. Then $f^*g$ is a $d$-bounded K\"ahler metric on $M$.
\end{proposition}

\begin{proof}
Let $\omega$ be the K\"ahler form of $g$ and let $h:=f^*g$. It is easy to check that in general the pullback through a smooth map commutes with the distributional action of the de Rham differential. Therefore the above statement  follows immediately noticing that if $\eta\in L^{\infty}\Omega^{1}(N,g)$  with $d\eta=\omega$ then $d(f^*\eta)=f^*\omega$ and $|f^*\eta|_h\leq |\omega|_g$ where $|\ |_h$ and $|\ |_g$ denote respectively the pointwise norm on $T^*M\otimes \mathbb{C}$ induced by $h$ and $g$.
\end{proof}

\begin{proposition}
\label{sum}
Let $M$ be a complex manifold and let $g$ and $h$ be two $d$-bounded K\"ahler metrics on $M$. Then the K\"ahler metric $\rho:=g+h$ is $d$-bounded as well. 
\end{proposition}

\noindent
In order to prove the above proposition we need the following elementary result.
\begin{proposition}
\label{dual}
Let $M$ be a manifold and let $g_1$ and $g_2$ be two Riemannian metrics on $M$ such that $g_2\leq g_1$.  Let $g_1^*$ be the metric that $g_1$ induces on $T^*M$ and analogously let $g_2^*$ be the metric that $g_2$ induces on $T^*M$. Then we have $g_1^*\leq g_2^*$.
\end{proposition}

\begin{proof}
Let $A\in \End(TM)$ such that $g_2(\cdot,\cdot )=g_1(A\cdot,\cdot )$. Then, for each $p\in M$, $A_p:T_p M\rightarrow T_p M$ is positive, symmetric 
with respect to $g_1$ and its eigenvalues are bounded above by 1. Let $A^{-1}$ be the inverse of $A$ and let $(A^{-1})^t\in \End(T^*M)$ be the transposed endomorphism of  $A^{-1}$.
An easy calculation of linear algebra shows that $g^*_2(\cdot ,\cdot )=g_1^*((A^{-1})^t \cdot,\cdot )$. Now, for each $p\in M$, $(A^{-1})^t_p:T^*_pM\rightarrow T^*_pM$ is positive, symmetric with respect to $g^*_1$ and its eigenvalues are bounded below by 1. This in turn implies immediately that  $g_1^*\leq g_2^*$ as 
required.
\end{proof}

\noindent
Now we give a proof of Prop. \ref{sum}

\begin{proof}
Let $\omega$ be the K\"ahler form of $g$ and analogously let  $\tau$ be the K\"ahler form of $h$. According to the assumptions there exist  $1$-forms $\alpha\in L^{\infty}\Omega^1(M,g)$,  $\beta\in L^{\infty}\Omega^1(M,h)$ such that $d\alpha = \omega$ and $d\beta=\tau$. Let us label by $\rho^*$, $g^*$ and $h^*$ the metrics on $T^*M\otimes \mathbb{C}$ induced respectively by $\rho$, $g$ and $h$. Clearly $g\leq \rho$ and $h\leq \rho$.  Then, using Prop. \ref{dual},  we have
\begin{align}
& \nonumber  (\rho^*(\alpha+\beta,\alpha+\beta))^{\frac{1}{2}}\leq (\rho^*(\alpha,\alpha))^{\frac{1}{2}}+(\rho^*(\beta,\beta))^{\frac{1}{2}}\leq (g^*(\alpha,\alpha))^{\frac{1}{2}}+(h^*(\beta,\beta))^{\frac{1}{2}}
\end{align}
Since $\alpha \in L^{\infty}\Omega^1(M,g)$ and $\beta \in L^{\infty}\Omega^1(M,h)$ we can  conclude that $\alpha+\beta\in L^{\infty}\Omega^1(M,\rho)$ as desired. 
\end{proof}

We go on by spending a few words about  resolution of singularities. We recall only what is strictly necessary for our purposes and we refer to the seminal work \cite{HH} and also to \cite{BM} for in-depth treatments of this topic. For a throughout discussion about  complex spaces we refer to the monographs  \cite{GFI} and \cite{GRRE}.\\

\medskip
Consider a compact and irreducible complex space $V$. Then, thanks to the fundamental work of Hironaka, we know that the singularities of $V$ can be resolved. More precisely there exists a compact complex manifold $M$, a divisor with only normal crossings $D\subset M$, a surjective and holomorphic map $\pi:M\rightarrow V$ such that $\pi^{-1}(\sing(V))=D$ and $\pi|_{M\setminus D}:M\setminus D\rightarrow \reg(V)$ is a biholomorphism. Moreover if $V\subset N$ is an analytic subvariety of a compact complex manifold then there exists a compact complex manifold $M$, a compact complex submanifold $Z\subset M$, a surjective  holomorphic map $\pi:M\rightarrow N$  and a divisor with only normal crossings $D\subset M$ such that $\pi^{-1}(\sing(V))=D$, $\pi|_{M\setminus D}:M\setminus D\rightarrow N\setminus \sing(V)$ is a biholomorphism and  $\pi|_{Z\setminus (Z\cap D)}:Z\setminus (Z\cap D)\rightarrow V\setminus \sing(V)$ is a biholomorphism. The latter is the so-called embedded desingularization.\\
 We introduce now some presheaves and the corresponding sheaves arising by sheafification.  Let $M$ be a compact complex manifold, $D\subset M$ a divisor with only normal crossings and $g$ any Hermitian metric on $M\setminus D$. Consider the  preasheaves $C^{p,q}_{D,g}$ on $M$ given by the assignments  
\begin{equation}
\label{presheaf}
C^{p,q}_{D,g}(U):=\{\mathcal{D}(\overline{\pa}_{p,q,\max})\ \text{on}\ (U\setminus U\cap D, g|_{U\setminus U\cap D})\};
\end{equation}
in other words to every open subset $U$ of $M$ we assign the maximal domain of $\overline{\pa}_{p,q}$ over $U\setminus (U\cap D)$ with respect to the Hermitian metric $g|_{U\setminus U\cap D}$.
The sheafification of $C^{p,q}_{D,g}$ is denoted by  $\cC^{p,q}_{D,g}$ and its sections over an open subset $U\subset M$ are
\begin{multline}\label{sheafifi}
\cC^{p,q}_{D,g}(U):=  \{s\in L^2_{\loc}\Omega^{p,q}(U\setminus U\cap D,g|_{U\setminus U\cap D})\ \text{such that for each }\ p\in U\ \text{there exists an}\\ \text{ open}\ \text{neighborhood}\ W\ \text{with}\ p\in W\subset\ U\ \text{such that}\ s|_{W\setminus W\cap D}\in  \mathcal{D}(\overline{\pa}_{p,q,\max})\ \text{on}\ (W\setminus W\cap D, g|_{W\setminus W\cap D})\}.
\end{multline}
We have now all the ingredients to state the main result of this section.

\begin{theorem}
\label{cohomology}
Let $V$ be a compact and irreducible complex space of complex dimension $v$. Let $\pi:M\rightarrow V$ be a  resolution of $V$ with $D:=\pi^{-1}(\sing(V))$ a normal crossings divisor in $M$. Let $h$ be a complete Hermitian metric on $\reg(V)$ and let $\sigma$ be the complete Hermitian metric on $M\setminus D$ defined as $\sigma:=(\pi|_{M\setminus D})^*h$. Assume that:
\begin{itemize}
\item $\pi_*\cC^{v,q}_{D,\sigma}$ is a fine sheaf for each $q=0,...,v$,
\item for each $p\in \sing(V)$ there is an open neighborhood $U$ and a  $d$-bounded K\"ahler metric $g_{U}$ on $\reg(U)$ such that $h|_{U}$ and $g_{U}$ are quasi-isometric. 
\end{itemize}
 Then we have the following isomorphism for each $q=0,...,v$:
$$H^{v,q}_{2,\overline{\partial}}(\reg(V),h)\cong H^{v,q}_{\overline{\partial}}(M).$$
\end{theorem}  

Before tackling  the proof we recall the following properties.

\begin{proposition}
\label{hilbertcomparison}
Let $M$ be a complex manifold of complex dimension $m$, let $(E,\rho)$ be a Hermitian  vector bundle on $M$  and let $g$ and $h$ be two Hermitian metrics on $M$. Then we have an equality of Hilbert spaces $$L^2\Omega^{m,0}(M,E,g)=L^2\Omega^{m,0}(M,E,h).$$ Assume now that $cg\geq h$ for some $c>0$. Then for each $q=1,...,m$ there exists a constant $\xi_q>0$ such that for every $s\in \Omega^{m,q}_c(M,E)$  we have 
\begin{equation}
\label{gamma}
\|s\|^2_{L^2 \Omega^{m,q}(M,E,g)}\leq \xi_q \|s\|^2_{L^2 \Omega^{m,q}(M,E,h)}.
\end{equation}
Therefore the identity $\Omega^{m,q}_c(M,E)\rightarrow \Omega^{m,q}_c(M,E)$ induces a continuous inclusion $$L^2 \Omega^{m,q}(M,E,h)\hookrightarrow L^2 \Omega^{m,q}(M,E,g)$$ for each $q=1,...,m$.
\end{proposition} 

\begin{proof}
The statement   follows  by the computations carried out in \cite{GMMI} pag. 145. 
\end{proof}

\noindent
We are now in the position to prove Th. \ref{cohomology}.

\begin{proof}
Let $\pi:M\rightarrow V$ be a resolution of $V$. Let $\mathcal{K}_{M}$ be the canonical sheaf of $M$, that is, the sheaf whose sections over any open subset $U$ of $M$ are the holomorphic $(n,0)$-forms over $U$. Let us consider the following sheaf $\mathcal{K}_V:=\pi_*\mathcal{K}_M$. This is the so-called Grauert-Riemenschneider canonical sheaf introduced in \cite{GRI}. By the Takegoshi vanishing theorem, see \cite{TaKe}, we get  that $H^q(M,\mathcal{K}_M)\cong H^q(V,\mathcal{K}_V)$ for each $q=0,...,v$, see for instance \cite{JRu} for the details.  We are therefore left with the task of
showing  that $H^{q}(V,\mathcal{K}_{V})\cong H^{n,q}_{2,\overline{\pa}}(\reg(V),h)$ for each $q=0,...,n$.  To this end consider the complex of sheaves $\{\pi_*\cC^{v,q}_{D,\sigma}, q\geq 0\}$, see \eqref{sheafifi}, whose morphisms are those induced by the distributional action of $\overline{\partial}_{v,q}$. It is clear that the cohomology groups of the complex given by the global sections of $\{\pi_*\cC^{v,q}_{D,\sigma}, q\geq 0\}$, that is  
\begin{equation}
\label{globalsections}
0\rightarrow \pi_*\cC^{n,0}_{D,\sigma}(V)\rightarrow...\rightarrow \pi_*\cC^{n,q}_{D,\sigma}(V)\rightarrow...\rightarrow \pi_*\cC^{n,n}_{D,\sigma}(V)\rightarrow 0
\end{equation}
are $H^{n,q}_{2,\overline{\pa}}(\reg(V),h)$, $q=0,...,n$. Therefore our goal is  to show that the complex $\{\pi_*\cC^{v,q}_{{D,\sigma}},q\geq 0\}$ is a fine resolution of $\cK_{V}.$ Since we assumed that $\pi_*\cC^{v,q}_{D,\sigma}$ is a fine sheaf for each $q=0,...,v$  we have only to prove that $\{\pi_*\cC^{v,q}_{{D,\sigma}}, q\geq 0\}$ is a resolution of $\cK_V$. We start to tackle this problem by showing that $\{\pi_*\cC^{v,q}_{{D,\sigma}}, q\geq 0\}$ is an exact sequence of sheaves.
Let $p$ be any point in $V$. It is clear that if $p\in \reg(V)$ then the induced sequence at the level of stalks, $\{(\pi_*\cC^{v,q}_{{D,\sigma}})_p, q\geq 0\}$, is exact. Hence we can assume that $p\in \sing(V)$. As a  first step in order to show that  $\{(\pi_*\cC^{v,q}_{{D,\sigma}})_p, q\geq 0\}$ is exact for any $p\in \sing(V)$ we need to introduce an auxiliary complete K\"ahler metric on a neighborhood of $p$ that satisfies the assumptions of Th. \ref{L2vanishing}. This is done as follows.\\ According to the assumptions  we know that there exists a sufficiently small open neighbourhood $U$ of $p$ such that the restriction of $h$ to the regular part of $U$ is quasi-isometric to a K\"ahler  metric $g_{U}$ which  is $d$-bounded. Now, taking $U$ even smaller if necessary, we can assume that there exists a positive constant $c$, an integer $n>v$ and a proper holomorphic embedding  $\phi: U \longrightarrow B(0,c)$ where $B(0,c)$ is the ball in $\bbC^n$ centered in $0$  with radius $c$. Let $\psi:B(0,c)\rightarrow \mathbb{R}$ be defined as $\psi:=-(\log(c^2-|z|^2))$ and let $g$ be the K\"ahler metric on $B(0,c)$ whose K\"ahler form is given by $\sqrt{-1}\partial \overline{\partial}\psi$. It is easy to check that $g$ satisfies the Ohsawa condition and therefore in particular is $d$-bounded, see for instance \cite{PS} pag. 613.  This in turn implies that $\rho_{U}:=(\phi|_{\reg(U)})^*g$ is a $d$-bounded K\"ahler metric on $U$, see Prop. \ref{pullback}. Now we introduce the  following K\"ahler metric on $\reg(U)$:
\begin{equation}
\label{sgamma}
\gamma_U:=\rho_{U}+g_{U}
\end{equation}

Next we check that $\gamma_U$ satisfies the hypothesis of Th. \ref{L2vanishing}. We begin by proving that $\gamma_U$ is complete.  According to Gordon's Theorem, \cite{Gor} Theorem 2, this is equivalent to showing the existence of a positive, smooth  and proper function  $f:\reg(U)\rightarrow \bbR$ with bounded gradient. Let $b:B(0,c)\rightarrow \bbR$ be a smooth function which satisfies the condition of Gordon's Thereom with respect to the complete K\"ahler metric associated to the $(1,1)$-form $\sqrt{-1}\partial \overline{\partial}\psi$. Let $\beta_U:\reg(U)\rightarrow \bbR$ be defined as $(b\circ \phi)|_{\reg(U)}$. Let $\tau:\reg(V)\rightarrow \bbR$ be a smooth function which satisfies the condition of Gordon's Theorem with respect to the  metric $h$. Let us label by $\tau_{U}$ the restriction of $\tau$ to $\reg(U)$. We claim that $\beta_{U}+\tau_{U}$ is smooth, proper and with bounded gradient with respect to $\gamma_U$. We first show that $\beta_{U}+\tau_{U}$ has bounded gradient with respect to $\gamma_U$. Labeling  by $\gamma^*_U$, $g_{U}^*$ and $\rho_{U}^*$ the metrics induced respectively by $\gamma_U$, $g_{U}$ and $\rho_{U}$ on  $T^*\reg(U)$, our task is equivalent to showing that $\beta_{U}+\tau_{U}$ has bounded differential with respect to $\gamma^*_U$. By Prop. \ref{dual} we have $|d\tau_U|_{\gamma^*_U}\leq |d\tau_{U}|_{g_{U}^*}$ and $|d\beta_U|_{\gamma^*_U}\leq |d\beta_U|_{\rho_{U}^*}$. Therefore we have: $$|d\tau_U+d\beta_U|_{\gamma^*_U}\leq |d\tau_U|_{\gamma^*_U}+|d\beta_U|_{\gamma^*_U}\leq |d\tau_U|_{g_U^*}+|d\beta_U|_{\rho_U^*}.$$ Finally $|d\tau_U|_{g_U^*}$ is bounded because $|d\tau|_{h^*}$ is bounded  and $g_U$ and $h|_{\reg(U)}$ are quasi-isometric on $\reg(U)$.  Analogously $|d\beta_U|_{\rho_U^*}$ is bounded because $\beta_U=(b\circ \phi)|_{\reg(U)}$, $b$ satisfy the conditions of Gordon's Theorem with respect to $g$ and $\rho_U=(\phi|_{\reg(U)})^*g$. Clearly $\beta_U+\tau_U$ is smooth. It remains to show that it is proper. However, this is clear
because it is easy to see, from the very definition of $\beta_U$ and $\tau_U$, that if $\{p_j\}$ is a sequence of points in $\reg (U)$ converging
to a point $p$ in $\overline{\reg{U}} \setminus \reg(U)$, then $(\beta_U+\tau_U)(p_j)\to +\infty$ as $j\to +\infty$. Furthermore, according to Prop. \ref{sum},  we know that $\gamma_U$ is a $d$-bounded K\"ahler metric because it is  defined as the sum of two $d$-bounded K\"ahler metrics. Hence, by Th. \ref{L2vanishing},  we can conclude that $$H^{v,q}_{2,\overline{\partial}}(\reg(U),\gamma_U)=0$$ for $q>0$.
Now, equipped with the above vanishing result, we can come back to the complex of sheaves $\{\pi_*\cC^{v,q}_{{D,\sigma}},q\geq 0\}$. Let $p\in \sing(V)$. In order to conclude that the complex  $\{(\pi_*\cC^{v,q}_{{D,\sigma}})_p,q\geq 0\}$ is exact, it is enough to show that given any open neighborhood $A$ of $p$  every cohomology class in $H^{v,q}_{2,\overline{\pa}_{\max}}(\reg(A),h|_{\reg(A)})$ admits a representative that becomes exact when restricted to some open subset $W\subset A$ with $p\in W$. This is done as follows.  Consider again any point $p\in \sing(V)$ and any open neighborhood $A$ of $p$. Let $[\nu]\in H^{v,q}_{2,\overline{\pa}_{\max}}(\reg(A),h|_{\reg(A)})$. According to  \cite{BLHC} Th. 3.5 we know that $[\nu]$ admits a smooth representative that we label by $\nu$.  Let $U\subset A$ be an open neighborhood  of $p$ such that $h|_{\reg(U)}$ is quasi isometric to a $d$-bounded K\"ahler metric $g_U$.  Clearly we have $h|_{\reg(U)}\leq \gamma_U$, where $\gamma_U$ is defined as in \eqref{sgamma}. Hence by Prop. \ref{hilbertcomparison} we get that $\nu|_{\reg(U)}\in \ker(\overline{\pa}_{v,q})\subset L^2\Omega^{v,q}(\reg(U),\gamma_U)$. Thus, according to what we have just shown above, there exists $\mu\in \cD(\overline{\partial}_{v,q-1})\subset L^2\Omega^{v,q-1}(\reg(U),\gamma_U) $ such that $\overline{\partial}_{v,q-1}\mu=\nu|_{\reg(U)}$ in $L^2\Omega^{v,q}(\reg(U),\gamma_U)$. Let now $W$ be any open subset of $U$ such that $\overline{W}\subset U$ and $p\in W$. It immediate to check that $\gamma_U|_{\reg(W)}$ and $h|_{\reg(W)}$ are quasi-isometric. Therefore we have $\mu|_{\reg(W)}\in \cD(\overline{\partial}_{v,q-1,\max})\subset L^2\Omega^{v,q-1}(\reg(W),h|_{\reg(W)}) $ and $\overline{\partial}_{v,q-1,\max}(\mu|_{\reg(W)})=\nu|_{\reg(W)}$ in $L^2\Omega^{v,q}(\reg(W),h|_{\reg(W)})$ and so we can conclude that $\{\pi_*\cC^{v,q}_{D,\sigma},q\geq 0\}$ is an exact sequence of sheaves as required.\\Finally we finish the proof of Th. \ref{cohomology} by showing that $\pi_* \mathcal{K}_{M}$ is equal to the kernel
of the morphism $\pi_*\cC^{v,0}_{D,\sigma}\to  \pi_*\cC^{v,1}_{D,\sigma}$ induced by
$\overline{\pa}_{v,0}.$ To this aim we work on $M$ and indeed we show that $\mathcal{K}_{M}$ is equal to the kernel of the morphism $\cC^{v,0}_{D,\sigma}\to  \cC^{v,1}_{D,\sigma}$ induced by $\overline{\pa}_{v,0}.$ Consider any open subset $U$ of $M$ and let  $p$ be any point in $U$. Let $\alpha\in \cC^{v,0}_{D,\sigma} (U)$ be  such that  $\alpha\in \ker (\cC^{v,0}_{D,\sigma}\to  \cC^{v,1}_{D,\sigma})$.
Hence, there exists an open neighbourhood $W$ of $p$, with closure contained in $U$, such that $\alpha$ restricted to $W\setminus
(W\cap D)$, $\alpha |_{W\setminus (W\cap D)}$, lies in $L^2\Omega^{v,0} (W\setminus (W\cap D),\sigma |_{W\setminus (W\cap D)})$
and satisfies 
\begin{equation}\label{holo}
\overline{\pa}_{v,0,\max}\, \alpha |_{W\setminus (W\cap D)}=0\,.
\end{equation}
 Thus in turn implies that
\begin{equation}
\label{maxmax}
(\overline{\pa}_{v,0,\max})^*( \overline{\pa}_{v,0,\max} \alpha |_{W\setminus (W\cap D)})=0.
\end{equation}
Consider now the complex $\{\Omega^{v,*}_c (W\setminus (W\cap D)),\overline{\pa}_{v,*}\}$; it is well known that this is an elliptic complex 
and thus the associated laplacians $\Delta_{\overline{\pa},v,q}$ are elliptic for each $q$. By \eqref{maxmax} we know in particular that  $\alpha |_{W\setminus (W\cap D)}$
is in the null space of the maximal extension of $\Delta_{\overline{\pa},v,0}:L^2\Omega^{v,0} (W\setminus (W\cap D),\sigma |_{W\setminus (W\cap D)})\rightarrow L^2\Omega^{v,0} (W\setminus (W\cap D),\sigma |_{W\setminus (W\cap D)})$. Hence by elliptic regularity we can conclude that $\alpha |_{W\setminus (W\cap D)}$  is smooth, and thus, by \eqref{holo},  holomorphic
on $W\setminus (W\cap D)$. Summarizing: $\omega$ lies in $L^2\Omega^{v,0} (W\setminus (W\cap D), \sigma | _{W\setminus (W\cap D)})$
and it is holomorphic on $W\setminus (W\cap D)$. Now, if $W\cap D=\emptyset$ we can already conclude that $\omega$ is homolorphic in all of $W$. If $W\cap D\neq \emptyset$ let  $\lambda$ be an arbitrary Hermitian metric on $M$ and let us consider
$\lambda |_{W\setminus (W\cap D)}$. According to Prop. \ref{hilbertcomparison} we know that
  $$L^2\Omega^{v,0} (W\setminus (W\cap D),\sigma | _{W\setminus (W\cap D)})=L^2\Omega^{v,0} (W\setminus (W\cap D),\lambda | _{W\setminus (W\cap D)})\,.
$$
Thus $\alpha$ is in $L^2\Omega^{v,0} (W\setminus (W\cap D),\lambda| _{W\setminus (W\cap D)})$
and it is holomorphic on $W\setminus (W\cap D)$. Finally using an $L^2$-extension theorem as in \cite{JRu}  we can conclude that $\alpha$  extends as an holomorphic $(v,0)$-form in all of $W$. Replacing  $p$ with any other point  in $U$ and repeating the same argument  we can  conclude that  $\alpha$ is holomorphic in $U$ and, therefore, that it is an element of $\mathcal{K}_{M} (U)$. Clearly the other inclusion is trivial, that is: $\cK_{M}$ is a sub-sheaf of the kernel of the morphism $\cC^{v,0}_{D,\sigma}\to  \cC^{n,1}_{D,\sigma}$ which is induced by the distributional action of $\overline{\pa}_{v,0}$. Indeed if $\omega\in \mathcal{K}_{M}(U)$ then $\omega\in \mathcal{C}^{v,0}_{D,\sigma}(U)$ and the distributional action of $\overline{\pa}_{v,0}$ applied  to $\omega$ is equal to 0. We can thus conclude that $\{\pi_*\mathcal{C}^{v,q}_{D,\sigma},\ q\geq 0\}$ is a fine resolution of $\pi_*\mathcal{K}_{M}$ as desired.
\end{proof}

We give now a criterion which assures that the sheaves $\{\pi_*\mathcal{C}^{v,q}_{D,\sigma},q\geq 0\}$ are fine. 
\begin{lemma}
\label{finecriterion}
In the setting of Th. \ref{cohomology}. Assume that  given any open cover $\cU=\{U_{i}\}_{i\in I}$ of $V$ there exists a continuous partition of unity $\{\lambda_j\}_{j\in J}$ subordinate to $\cU$ such that for each $j\in J$ 
\begin{enumerate}
\item $\lambda_j|_{\reg(V)}$ is smooth
\item $\|d(\lambda_j|_{\reg(V)})\|_{L^{\infty}\Omega^1(\reg(V),h)}<\infty$.
\end{enumerate}
Then $\pi_*\mathcal{C}^{v,q}_{D,\sigma}$ is a fine sheaf for each $q=0,...,v$.
\end{lemma}

\begin{proof}
This follows immediately from the description given in \eqref{sheafifi}.
\end{proof}

We proceed by showing some corollaries of Th. \ref{cohomology}.

\begin{corollary}
\label{fred1}
The unique closed extension of the operator $\overline{\pa}_{v} + \overline{\pa}^t_{v} : \Omega^{v,\bullet}_c (\reg(V),h)\to 
\Omega^{v,\bullet}_c (\reg(V),h) $, denoted here
\begin{equation}
\label{highland}
\overline{\pa}_{v} + \overline{\pa}^*_{v} : L^2 \Omega^{v,\bullet} (\reg(V),h)\to L^2 \Omega^{v,\bullet} (\reg(V),h) 
\end{equation}
is a Fredholm operator on its domain endowed with the graph norm.
\end{corollary}

\begin{proof}
This follows immediately from the  finite dimensionality of  $H^{v,q}_{2,\overline{\pa}} (\reg(V), h)$ and Theorem 2.4 in \cite{BLHC}.
\end{proof}

\begin{corollary}
\label{fred2}
For each $q=0,...,v$ we have the following isomorphism: $$H^{0,q}_{2,\overline{\pa}}(\reg(V),h)\cong H^{0,q}_{\overline{\pa}}(M).$$ In particular we have $$\chi(M,\cO_M)=\chi_2(\reg(V),h)$$ where the term on the right-hand side  is defined as $\sum (-1)^q\dim(H^{0,q}_{2,\overline{\pa}}(\reg(V),h))$.
\end{corollary}

\begin{proof}
This follows from Th. \ref{cohomology} and the $L^2$ version of  Serre duality, see \cite{JRu}, which tells us that $H^{v,q}_{2,\overline{\partial}}(\reg(V),h)\cong H^{0,q}_{2,\overline{\partial}}(\reg(V),h)$. 
\end{proof}

\begin{corollary}
\label{fred3}
The unique closed extension of the operator $\overline{\pa}_{0} + \overline{\pa}^t_{0} : \Omega^{0,\bullet}_c (\reg(V),h)\to 
\Omega^{0,\bullet}_c (\reg(V),h) $, denoted here
$$\overline{\pa}_{0} + \overline{\pa}^*_{0} : L^2 \Omega^{0,\bullet} (\reg(V),h)\to L^2 \Omega^{0,\bullet} (\reg(V),h) $$
is a Fredholm operator on its domain endowed with the graph norm.
\end{corollary}

\begin{proof}
As for Cor. \ref{fred1} this follows immediately  by the finite dimensionality of  $H^{0,q}_{2,\overline{\pa}} (\reg(V),h)$ and Theorem 2.4 in \cite{BLHC}.
\end{proof}

\begin{corollary}
\label{fred4}
In the setting of Th. \ref{cohomology}. Assume moreover that $h$ is K\"ahler. Then  $H^{q,0}_{2,\overline{\pa}}(\reg(V),h)$ is finite dimensional for each $q=0,...,v$. Moreover $$\overline{\pa}_{q,0}:L^2\Omega^{q,0}(\reg(V),h)\rightarrow L^2\Omega^{q,0}(\reg(V),h)$$ that is the unique closed extension of $\overline{\pa}_{q,0}:\Omega^{q,0}_c(\reg(V))\rightarrow \Omega^{q,1}_c(\reg(V))$, has closed range.
\end{corollary}
\begin{proof}
We consider $ \Delta_{v,q,\overline{\pa}}:\Omega^{v,q}_c(\reg(V))\rightarrow \Omega^{v,q}_c(\reg(V))$,
with $\Delta_{v,q,\overline{\pa}}:=\overline{\pa}^t_{v,q} \overline{\pa}_{v,q}+\overline{\pa}_{v,q}
\overline{\pa}^t_{v,q}$. According to Th. 2.4 in \cite{BLHC} we know that the unique closed extension of this operator, labeled here by 
\begin{equation}
\label{giove}
\Delta_{v,q,\overline{\pa}}:L^2\Omega^{v,q}(\reg(V),h)\rightarrow L^2\Omega^{v,q}(\reg(V),h),
\end{equation} is a Fredholm operator on its domain endowed with the graph norm. 
Consider now   
\begin{equation}
\label{venere}
\Delta_{v-q,0,\pa}:L^2\Omega^{v-q,0}(\reg(V),h)\rightarrow L^2\Omega^{v-q,0}(\reg(V),h)
\end{equation}
 that is, the operator defined as the unique closed extension of $\Delta_{v-q,0,\pa}$ on $\Omega^{v-q,0}_c(\reg(V))$.
 %\rightarrow \Omega^{v-q,0}_c(\reg(V))$. 
 It is easy to see that the Hodge star operator $*:L^2\Omega^{v,q}(\reg(V),h)\rightarrow L^2\Omega^{v-q,0}(\reg(V),h)$ makes \eqref{giove} and \eqref{venere} unitarily equivalent. On the other hand since $(\reg(V),h)$ is K\"ahler we have $\Delta_{v-q,0,\pa}=\Delta_{v-q,0,\overline{\pa}}$ on $\Omega^{v-q,0}(\reg(V))$. This in turn implies that we have an equality of operators acting on $L^2\Omega^{v-q,0}(\reg(V),h)$ that is $\Delta_{v-q,0,\pa}:L^2\Omega^{v-q,0}(\reg(V),h)\rightarrow L^2\Omega^{v-q,0}(\reg(V),h)$ coincides with $\Delta_{v-q,0,\overline{\pa}}:L^2\Omega^{v-q,0}(\reg(V),h)\rightarrow L^2\Omega^{v-q,0}(\reg(V),h)$. Hence we can conclude that $\Delta_{v-q,0,\overline{\pa}}:L^2\Omega^{v-q,0}(\reg(V),h)\rightarrow L^2\Omega^{v-q,0}(\reg(V),h)$ is a Fredholm operator on its domain endowed with the graph norm. Moreover the completness of $(\reg(V),h)$ assures us that $\Delta_{v-q,0,\overline{\pa}}:L^2\Omega^{v-q,0}(\reg(V),h)\rightarrow L^2\Omega^{v-q,0}(\reg(V),h)$ coincides with $$\overline{\pa}_{v-q,0}^*\circ\overline{\pa}_{v-q,0}:L^2\Omega^{v-q,0}(\reg(V),h)\rightarrow L^2\Omega^{v-q,1}(\reg(V),h)$$ where $\overline{\pa}_{v-q,0}:L^2\Omega^{v-q,0}(\reg(V),h)\rightarrow L^2\Omega^{v-q,1}(\reg(V),h)$ is the unique closed extension of $\overline{\pa}_{v-q,0}:\Omega^{v-q,0}_c(\reg(V))\rightarrow \Omega^{v-q,1}_c(\reg(V))$ and $\overline{\pa}_{v-q,0}^*:L^2\Omega^{v-q,1}(\reg(V),h)\rightarrow L^2\Omega^{v-q,0}(\reg(V),h)$ is its Hilbert space adjoint. In conclusion this shows that on $L^2\Omega^{v-q,0}(\reg(V),h)$ we have $$\ker(\Delta_{v-q,0,\overline{\pa}})=\ker(\overline{\pa}_{v-q,0})=H^{v-q,0}_{2,\overline{\pa}}(\reg(V),h).$$ We can thus conclude that $H^{v-q,0}_{2,\overline{\pa}}(\reg(V),h)$ is finite dimensional. Finally  that  $$\overline{\pa}_{v-q,0}:L^2\Omega^{v-q,0}(\reg(V),h)\rightarrow L^2\Omega^{v-q,1}(\reg(V),h)$$ has closed range can be easily shown using the fact that  $\Delta_{v-q,0,\overline{\pa}}:L^2\Omega^{v-q,0}(\reg(V),h)\rightarrow L^2\Omega^{v-q,0}(\reg(V),h)$ is a Fredholm operator on its domain endowed with the graph norm arguing  for instance as in \cite{FB} Cor. 4.1.
\end{proof}

\noindent
We end this section with the following consequence.
\begin{corollary}
\label{bir}
Let $V$ and $W$ be a pair of compact and irreducible complex spaces of complex dimension $v$ whose regular parts are endowed with complete Hermitian metrics $h$ and $g$ respectively. Assume that both $(V,h)$ and $(W,g)$ satisfy the assumptions of Th. \ref{cohomology}. Assume moreover that $V$ and $W$ are bimeromorphic. Then for each $q=0,...,v, $
\begin{equation*}
 H^{v,q}_{2,\overline{\partial}}(\reg(V,h))\cong H^{v,q}_{2,\overline{\partial}}(\reg(W,g))\\
\quad \quad \quad H^{0,q}_{2,\overline{\partial}}(\reg(V,h))\cong H^{0,q}_{2,\overline{\partial}}(\reg(W,g)).
\end{equation*}
\end{corollary}

\begin{proof}
Let $\pi:M\rightarrow V$ be a resolution of $V$ and analogously let $p:N\rightarrow W$ be a resolution of $W$. Since we assumed that $V$ and $W$ are bimeromorphic we have $H^{v,q}_{\overline{\pa}}(M)\cong H^{v,q}_{\overline{\pa}}(N)$ for each $q=0,...,v$. Now the conclusion follows by Th. \ref{cohomology} and Cor. \ref{fred2}.
\end{proof}

\section{Applications}
The aim of this section is to collect some  examples of Hermitian complex spaces where  Th. \ref{cohomology} and its corollaries  can be applied. In the first part we treat Saper-type K\"ahler metrics while in last part we consider complete K\"ahler manifolds with finite volume and pinched negative sectional curvatures. 
\subsection{Saper metrics} These kind of metrics have been defined in \cite{GMM} and \cite{GMMI} in order to extend to the case of arbitrary analytic subvarieties of K\"ahler manifolds the construction carried out by Saper in \cite{LS} in the setting of isolated singularities.\\In the next definition we  recall from \cite{GMM} pag. 741 the definition of Saper-type metric.\\
\begin{definition}
Let $V$ be a singular subvariety of a compact complex manifold $M$ and let $\omega$ be the fundamental $(1,1)$-form of a Hermitian metric on $M$. Let $\pi:\wt{M}\rightarrow M$ be a holomorphic map of a compact complex manifold $\wt{M}$ to $M$ whose exceptional set $E$ is a divisor with normal crossing in $\wt{M}$ and such that the restriction $$\pi|_{\wt{M}\setminus E}:\wt{M}\setminus E \longrightarrow M\setminus \sing(V)$$ is a biholomorphism. Let $L_E$ be the line bundle on $\wt{M}$ associated to $E$ and let $h$ be a Hermitian metric on $L_E$. Let $s:\wt{M}\rightarrow L_E$ be a global holomorphic section whose  associated divisor $(s)$ equals $E$ (in particular $s$ vanishes exactly on $E$) and let $\|s\|_h$ be the norm of $s$ with respect to $h$.\\A metric on $\wt{M}\setminus E$ which is quasi-isometric to a metric with fundamental $(1,1)$-form $$l\pi^*\omega-\frac{\sqrt{-1}}{2\pi}\pa\overline{\pa}\log(\log\|s\|_h^2)^2$$ for $l$ a positive integer, will be called a \textbf{Saper-type metric}, distinguished with respect to the map $\pi$. The corresponding metric on $M\setminus {\sing{V}}\cong \wt{M}\setminus E$ and its restriction to $V\setminus \sing{V}$ are also called Saper-type metric.  
\end{definition}
We follow the convention used in \cite{GMM}: thus, when it is clear form the context, we omit the 
sentence "distinguished with respect to $\pi$".

Now we go on stating a fundamental existence result for Saper-type metrics proved by  Grant Melles and Milman  in  \cite{GMM}, Theorem 8.6 pag. 746. 

\begin{theorem}
\label{existencesaper}
Let $V$ be a singular subvariety of a compact K\"ahler manifold $M$ and let $\omega$ be the K\"ahler $(1,1)$-form of a K\"ahler metric on $M$. There exists a $C^{\infty}$ function $F$ on $M$, vanishing on $\sing(V)$, such that the $(1,1)$-form 
\begin{equation}
\label{saper}
\omega_S=\omega-\frac{\sqrt{-1}}{2\pi}\pa\overline{\pa}\log(\log F)^2
\end{equation} 
is the K\"ahler form of a complete Saper-type metric on $M\setminus \sing(V)$ and hence on $V\setminus \sing(V)$. Furthermore the function $F$ can be constructed to be of the form $$F=\prod_{\alpha}F^{\rho_{\alpha}}_{\alpha}$$ where $\{\rho_{\alpha}\}$ is a $C^{\infty}$ partition of unity subordinate to an open cover $\{U_{\alpha}\}$ of $M$, $F_{\alpha}$ is a function on $U_{\alpha}$ of the form $$F_{\alpha}=\sum_{j=1}^r|f_j|^2$$ and $f_1,...,f_r$ are holomorphic functions on $U_{\alpha}$, vanishing exactly on $U_{\alpha}\cap \sing(V)$. More specifically $f_1,...,f_r$ are local 
holomorphic generators of a coherent ideal sheaf $\cI$ on $M$ such that blowing up $M$ along $\cI$  desingularizes $V$, 
$\cI$ is supported on $\sing(V)$ and the exceptional divisor in the blow-up $\wt{M}$ along $\cI$  has normal crossing and 
has also normal crossing with the strict transform $\wt{V}$ of $V$ in $\wt{M}$ (the so called embedded desingularization of $X$.) 
\end{theorem}

Concerning Saper-type metric we have the following property.

\begin{proposition}
\label{Ohsawasaper}
In the setting of Th. \ref{existencesaper}. For each $p\in \sing(V)$ there exists an open neighborhood $U$ and a K\"ahler metric $g_U$ such that $g_S|_U$ is quasi-isometric to $g_U$ and $g_U$ satisfies the Ohsawa condition.
\end{proposition}
\begin{proof}
This follows from \cite{GMM} Prop. 8.10 and Prop. 9.11.
\end{proof}

We have now all the ingredients to apply Th. \ref{cohomology} to Saper-type K\"ahler metrics.
\begin{theorem}
\label{sapercohomology}
Let $M$ be a compact K\"ahler manifold with K\"ahler form $\omega$ and let $V$ be
an analytic subvariety  in $M$ of complex dimension $v$. Let  $\pi:\wtV\rightarrow V$ be a resolution of $V$. Finally let $g_S$ be  a Saper-type metric  on $\reg (V)$ as constructed in Th. \ref{existencesaper}. 
Then  the following isomorphism holds:
\begin{equation}
\label{mainiso}
H^{v,q}_{2,\overline{\pa}} (\reg (V), g_S)\cong H^{v,q}_{\overline{\pa}} (\widetilde{V})\,
\end{equation}
for every $q=v,...,n.$
\end{theorem}
\begin{proof}
Let $D\subset \widetilde{V}$ be the divisor with only normal crossing such that $\pi^{-1}(\sing(V))=D$. Let $\sigma_S:=(\pi|_{\widetilde{V}\setminus D})^*g_S$. Thanks to Prop. \ref{Ohsawasaper}, in order to deduce the above theorem by Th. \ref{cohomology}, we have only to check that the sheaves $\{\pi_*\cC_{D,\sigma_S}^{v,q},q\geq 0\}$  are fine. This is done as follows. In order to prove that $\pi_*\cC^{v,q}_{D,\sigma_S}$ is fine it is enough to show that given a cover $\mathcal{U}=\{U_i\}_{i\in I}$ of $V$ there exists a continuous partition of unity $\{f_{\gamma}\}_{\gamma\in G}$ subordinated to $\mathcal{U}$ such that for each $\gamma\in G$
\begin{itemize}
\item $f_{\gamma}|_{\reg(V)}$ is smooth
\item If $A$ is an open subset of $V$ and $\omega\in \pi_*\cC^{v,q}_{D,g_S}(A)$ then $f_{\gamma} \omega \in \pi_*\cC^{v,q}_{D,g_S}(A)$.
\end{itemize}
To this end we recall    \cite[Proposition 10.2.1]{GMMI}:

\begin{proposition}
\label{fine}
Let $p\in \sing(V)$, let $U$ be an open neighborhood of $p$ in $V$,  let $f$ be a smooth function on $M$ and let $\omega\in L^2\Omega^k(\reg(U), g_S|_{\reg(U)})$. Then $d(f|_{\reg(U)})\wedge \omega\in  L^2\Omega^{k+1}(\reg(U), g_S|_{\reg(U)})$.
\end{proposition}

\noindent Clearly $d(f|_{\reg(U)})\wedge \omega$ $=\pa(f|_{\reg(U)})\wedge \omega+\overline{\pa}(f|_{\reg(U)})\wedge \omega$. Therefore, thanks to Prop. \ref{fine}, if  $\omega\in L^2\Omega^{p,q}(\reg(U),g_S|_{\reg(U)})$, $p+q=k$, we can conclude that $\overline{\pa}(f|_{\reg(U)})\wedge \omega\in L^2\Omega^{p,q+1}(\reg(U), g_S|_{\reg(U)})$. In particular if $\omega\in \pi_*\cC^{v,q}_{D,g_S}(U)$ then we can conclude that $f|_{\reg(U)}\omega \in \pi_*\cC^{v,q}_{D,g_S}(U)$.
Let now $\mathcal{V}=\{V_j\}_{j\in J}$ be an open cover of $M$ such that the induced cover on $V$ is equal to $\mathcal{U}$.
Considering a smooth partition of unity subordinated to $\mathcal{V}$ and restricting it to $V$ and using the remark that we have just made,
it is clear that we have built a partition of unity on $V$ subordinated to $\mathcal{U}$ with the required properties.
\end{proof}

\begin{remark}
In the particular case of isolated singularities the isomorphism above was established by Saper as  a corollary of 
his main result in \cite{LS}, namely the isomorphism of $H^\ell_{2} (\reg (V),g_S)\cong I^{\underline{m}}H^\ell (V,\RR)$, and the 
identification of the induced Hodge structure on $I^{\underline{m}}H^\ell (V,\RR)$ with the one constructed by Saito. Our proof, on the other hand,  is direct and rests solely on analytic arguments.
\end{remark}

We conclude this subsection with the following corollaries. 
\begin{corollary}
In the setting of Th. \ref{sapercohomology}. Then Cor. \ref{fred1}--Cor. \ref{fred4} hold for $(\reg(V),g_S)$.
\end{corollary}

\begin{corollary}
Let $(M,h)$ and $(N,g)$ be compact K\"ahler manifolds and let $V\subset M$, $W\subset N$ be analytic subvarieties of complex dimension $v$. Assume that $V$ and $W$ are bimeromorphic. Then for each $q=0,...,v$ 
\begin{equation*}
 H^{v,q}_{2,\overline{\pa}}(\reg(V),h_S)\cong H^{v,q}_{2,\overline{\pa}}(\reg(W),g_S)\\
\quad\quad\quad H^{0,q}_{2,\overline{\pa}}(\reg(V),h_S)\cong H^{0,q}_{2,\overline{\pa}}(\reg(W),g_S)
\end{equation*}
 where $h_S$ and $g_S$ are Saper-type metrics as constructed in Th. \ref{existencesaper} on $V$ and $W$ respectively.
\end{corollary}

\begin{proof}
This follows from Cor. \ref{bir}.
\end{proof}

\subsection{Further remarks on Saper metrics}
In this subsection we collect some byproducts of Th. \ref{cohomology} in the framework of Saper-type metrics. Consider again the setting of Theorem \ref{sapercohomology}.  To avoid any  confusion with the notations let us now label by $\omega'_S$ the $(1,1)$-form on $M\setminus \sing(V)$ given by $$\omega'_S=\omega-\frac{\sqrt{-1}}{2\pi}\pa\overline{\pa}\log(\log F)^2$$ where $\omega$ is the fundamental form of a K\"ahler metric on $M$ and $F$ is defined in \eqref{saper}. Let $g'_S$ be the Saper-type metric on $M\setminus \sing(V)$ whose fundamental form is $\omega'_S$. If we label by $i:\reg(V)\rightarrow M$ the inclusion of $\reg(V)$ in $M$, then we have  $g_S=i^*(g'_S)$ where $g_S$ is the Saper-type metric considered in Theorem \ref{sapercohomology}.  We have the following
\begin{theorem}
In the setting described above. We have the following isomorphisms:
\begin{equation}
\label{byproduct}
H^{m,q}_{2,\overline{\pa}}(M\setminus \sing(V),g'_S)\cong H^{m,q}_{\overline{\pa}}(M)
\end{equation}
\begin{equation}
\label{byproduct2}
H^{0,q}_{2,\overline{\pa}}(M\setminus \sing(V),g'_S)\cong H^{0,q}_{\overline{\pa}}(M)
\end{equation}
where $m$ is the complex dimension of $M$.
\end{theorem}

\begin{proof}
The proof is similar to the  one given for Th. \ref{sapercohomology} and for the sake of brevity we omit the details.
%\blue{$\longrightarrow$} 
Let us label by $V_s$  the singular locus of $V$. Consider  on $M$ the preasheaf $C^{p,q}_{V_s,g'_S}$ defined by $C^{p,q}_{V_s,g'_S}(U):=\{\mathcal{D}(\overline{\pa}_{p,q,\max})\ \text{on}\ (U\setminus U\cap V_s, g'_S|_{U\setminus U\cap V_s})\}$ where $U$ is any open subset of $M$.
The sheafification of $C^{p,q}_{V_s,g'_S}$ is denoted by  $\cC^{p,q}_{V_s,g'_S}$. Analogously to the proof of Theorem \ref{cohomology} we want to show that the complex $\{\cC^{m,q}_{V_s,g'_S}, q\geq 0\}$, whose  morphisms are induced by the distributional action of $\overline{\pa}_{m,q}$,  is a fine resolution of $\cK_M$. In particular that $\cC^{p,q}_{V_s,g'_S}$ is fine for each $p,q$ follows using a partition of unity of $M$. %Indeed if $h$ is any Hermitian metric on $M$ then for a suitable positive constant $c$ we have $g'_S\geq ch$, see \cite{GMM} pag. 744. According to  Prop. \ref{dual} this  tells us that the differential of any function involved in the partition of unity is bounded with respect to $g'_S$ because it is bounded with respect to $h$. 
Now let $p\in V_s$ and let $W$ be a sufficiently small neighborhood of $p$. We can assume that there exists a positive constant $c$, an  and a biholomorphism  $\phi: W \longrightarrow B(0,c)$ where $B(0,c)$ is the ball in $\bbC^m$ centered in $0$  with radius $c$. Let $\psi:B(0,c)\rightarrow \mathbb{R}$ be defined as $\psi:=-(\log(c^2-|z|^2))$, let $g$ be the K\"ahler metric on $B(0,c)$ whose K\"ahler form is given by $\sqrt{-1}\partial \overline{\partial}\psi$ and let  $\rho_{W}:=(\phi|_{\reg(W\setminus (W\cap V_s))})^*g$. By \cite{GMM} Prop. 8.10 and 9.11 we know that $g_{S'}|_{W\setminus (W\cap V_s)}$ is quasi isometric to a K\"ahler metric $g_W$ which satisfies the Ohsawa condition. Now we introduce the  following K\"ahler metric on $W\setminus (W\cap V_s)$:
\begin{equation}
\label{wgamma}
\gamma_W:=\rho_{W}+g_{W}
\end{equation}
Using $\gamma_{W}$ and arguing as in the proof of  Th. \ref{cohomology}  we can conclude that %$H^{m,q}_{2,\overline{\pa}_{\max}}(W\setminus W\cap V_s,g'_S|_{W\setminus W\cap V_s})=0$ for $q>0$ and this in turns  implies that 
$\{\cC^{m,q}_{V_s,g'_S},\ q\geq 0\}$ is an exact sequence of sheaves. Finally we are left to show  that the kernel of the sheaves morphism $\cC^{m,0}_{V_s,g'_S}\stackrel{\overline{\pa}_{m,0}}{\rightarrow}\cC^{m,1}_{V_s,g'_S}$ is $\cK_{M}$. This can be seen again as in the proof of Th. \ref{cohomology}. In conclusion $\{\cC^{m,q}_{V_s,g'_S}, q\geq 0\}$ is a fine resolution of $\cK_{M}$ and thus we can conclude that $H^{m,q}_{2,\overline{\pa}}(M\setminus \sing(V),g'_S)\cong H^{m,q}_{\overline{\pa}}(M)$. Applying  $L^2$-Serre duality we finally get $H^{0,q}_{2,\overline{\pa}}(M\setminus \sing(V),g'_S)\cong H^{0,q}_{\overline{\pa}}(M)$ and this completes the proof. %\blue{$\longleftarrow$} \\
%\red{Ho commentato la dimostrazione; penso che per PAMS sia meglio avere un articolo molto breve.}
\end{proof}

\medskip

Next we consider a Hermitian holomorphic line bundle $L$ on the resolution $\wtV$. We can restrict $L$ to $\wtV\setminus D$
and push it forward to $\pi_* L$ on $V\setminus \sing(V)$ through the biholomorphism $\pi |_{\wtV\setminus D}: \wtV\setminus D\to V\setminus \sing(V)$
(where $\pi: \wtV\to V$ is the map appearing in the resolution of $V$):
$$\pi_* L : = (\pi |_{\wtV\setminus D})^{-1})^* L |_{\wtV\setminus D}\,.$$
We shall say that $L$ is semipositive with respect to the base $V$ if for each $p\in V$ there exists a neighbourhood $U_p$ in $V$ such that 
$L$ is positive on $\pi^{-1} (U_p)$.

\begin{theorem}\label{theo:with-line-bundle}
Consider the setting of Theorem \ref{sapercohomology}. Let  $L\to \wtV$ be  a Hermitian holomorphic line bundle which is semipositive with respect to $V$.  Then there exist isomorphisms:
\begin{equation}\label{with-line-bundle}
H^{v,q}_{2,\overline{\pa}}(V\setminus \sing(V),\pi_* L, g_S)\cong H^{v,q}_{\overline{\pa}}(\wtV,L)\quad\quad
H^{0,q}_{2,\overline{\pa}}(V\setminus \sing(V),\pi_* L^*, g_S)\cong H^{0,q}_{\overline{\pa}}(\wtV,L^*)\,.
\end{equation}
\end{theorem}

\begin{proof}
Also in this case the proof is similar to the one   given for Theorem \ref{sapercohomology} but with some modifications due to the presence of the line bundle $L$. 
These modifications are as follows: first of all we can introduce with self-explanatory notation  the complex of sheaves 
$\{ \cC^{v,q}_{D,g_S,L}\}$ on $\tilde{V}$. As showed in \cite{JRu} we have $H^q(\tilde{V},\cK_{\tilde{V}}(L))=H^q(V,\pi_*\cK_{\tilde{V}}(L))$ where $\cK_{\tilde{V}}(L)$ is the sheaf of holomorphic sections of the holomorphic line bundle $K_{\tilde{V}}\otimes L$. Hence, as in the proof of Th. \ref{sapercohomology}, our purpose now is to show that $\{\pi_* \cC^{v,q}_{D,g_S,L}\}$ is  a fine resolution of $\pi_*(\cK_{\wtV}(L))$. To this aim, using a results proved by Ruppental, see \cite[Theorem 3.2]{JRu}, we know that for K\"ahler metrics satisfying the Ohsawa condition we can extend Th. \ref{L2vanishing} to the $L^2$-$\overline{\pa}$-cohomology of forms with bi-degree $(v,q)$ and  with coefficients in any semipositive Hermitian holomorphic line bundle. Clearly since $L$ is semipositive with respect to $V$ it obeys the conditions of Th. 3.2 in \cite{JRu}. Therefore, with an analogous strategy to  the one used in the proof of Theorem \ref{cohomology}, this vanishing result can in turn be used  in order to show that   $\{\pi_* \cC^{v,q}_{D,g_S,L}\}$ is  a fine resolution of $\pi_*(\cK_{\wtV}\otimes L) $. All this gives us the first isomorphism in \eqref{with-line-bundle}; using the $L^2$-version of Serre duality we get the second isomorphism.
\end{proof}

\medskip
For the next result we begin by recalling that a Hermitian holomorphic line bundle $L$ over %\footnote{irreducible ??
%In general, discuss carefully these issues...} 
an irreducible compact complex space $V$ is {\em almost positive} if the curvature form is
semipositive on $\reg(V) $ and positive on an open subset of $\reg (V)$. %\footnote{Clarify what is a holomorphic 
%line bundle over a singular complex space.}

\begin{theorem}\label{theo:vanishing}
Let  $\pi:\wtV\rightarrow V$ be as in Th. \ref{sapercohomology}. Let $L$ be an almost positive Hermitian holomorphic line bundle over $V$. Then for $q>0$ we have 
\begin{equation}\label{vanishing}
H^{v,q}_{2,\overline{\pa}}(V\setminus \sing(V),L, g_S)\cong
H^{0,v-q}_{2,\overline{\pa}}(V\setminus \sing(V), L^*, g_S)=0
\end{equation}
\end{theorem}

\begin{proof}
Let us define $F:=\pi^*L$. Then $F$ is an almost positive line bundle over $\wtV$. Using the Bochner-Kodaira-Nakano inequality, see \cite{JPDE} 13.3, and the fact that $F$ is positive on an open subset of $\wtV$, we easily get $H^{v,q}_{\overline{\pa}}(\wtV,F)=0$ for $q>0$. Now the conclusion  follows immediately by applying  Theorem \ref{theo:with-line-bundle}.
\end{proof}

\subsection{Negatively curved K\"ahler manifolds with finite volume}
Let $(M,h)$ be a complete K\"ahler manifold with finite volume and pinched negative sectional curvatures $-b^2\leq \sec_h\leq -a^2$ for some constants $0<a\leq b$. An important result concerning the geometry of such manifolds is the one proved in  \cite{SiuYau} by Siu and Yau. This result  provides the existence of a  compactification of $M$ in terms of a complex projective variety with only isolated singularities. More precisely if $(M,h)$ is a K\"ahler manifold as above then there exists a projective variety $V\subset \mathbb{C}\mathbb{P}^n$ with only isolated singularities such that $\reg(V)$ and $M$ are biholomorphic. The purpose of this subsection is to investigate the $L^2$-$\overline{\partial}$-cohomology of such K\"aher manifolds with the help of our Th. \ref{cohomology} and the Siu-Yau compactification. Concerning this task the main result of this subsection reads as follows:
\begin{theorem}
\label{negative}
Let $(M,h)$ be a complete K\"ahler manifold of complex dimension $m$ with finite volume. Assume that the sectional curvatures of $(M,h)$ satisfies  $-b^2\leq \sec_h\leq -a^2$ for some constants $0<a\leq b$. Let $V\subset \mathbb{C}\mathbb{P}^n$ be the Siu--Yau compactification of $M$ and let $\pi:\tilde{V}\rightarrow V$ be a resolution of $V$. Then we have the following isomorphism for each $q=0,...,m$ $$H^{m,q}_{2,\overline{\pa}}(M,h)\cong H^{m,q}_{\overline{\pa}}(\tilde{V}).$$ 
\end{theorem}
\begin{proof}
Let $\psi:M\rightarrow \reg(V)$ be a biholomorphism between $M$ and $\reg(V)$. Let us label by $\upsilon$ the K\"ahler metric $(\psi^{-1})^*h$. Henceforth we will indentify $(M,h)$ and $(\reg(V),\upsilon)$. According to \cite{NY} Lemma 3.2 we know that there exists a compact subset $D\subset M$ and a bounded continuous  $1$-form $\theta$ such that on $M\setminus D$ we have $d\theta=\omega$ where $\omega$ is the K\"ahler form of $h$. Hence the second condition in the statement of Th. \ref{cohomology} is fulfilled. We are left to show that the sheaves $\{\pi_*\cC^{m,q}_{D,\eta}, q\geq 0\}$ are fine where $D\subset \tilde{V}$ is the divisor with only normal crossings given by $D=\pi^{-1}(\sing(V))$ and $\eta:=(\pi|_{\tilde{V}\setminus D})^*\upsilon$.  Let $\cU:=\{U_i\}_{i\in I}$ be an open cover of $V$.  Since $V$ is compact there exists a finite open cover of $V$, $\cW:=\{W_1,...,W_r\}$ for some positive integer $r$, such that $\cW$ is subordinate to $\cU$ and such that for any $i,j\in \{1,...,r\}$ with $i\neq j$ we have $\sing(V)\cap W_i\cap W_j=\emptyset$. Now we can easily construct a partition of unity $\{\phi_1,...,\phi_r\}$ subordinated to $\cW$ such that the following properties hold:
\begin{itemize}
\item $\phi_i:V\rightarrow [0,1]$ is continuous for each $i=1,...,r$
\item $\phi_i|_{\reg(V)}:\reg(V)\rightarrow [0,1]$ is smooth for each $i=1,...,r$
\item if $p\in \sing(V)\cap \supp(\phi_i)$ then there exists a neighborhood $A$ of $p$ which is open in $V$ and  such that $\phi_i|_A=1$.
\end{itemize}
It is immediate to check that $\|d\phi_i\|_{L^{\infty}\Omega^1(\reg(V),\upsilon)}<\infty$. Hence by Lemma \ref{finecriterion} we can conclude that $\{\pi_*\cC^{m,q}_{D,\eta}, q\geq 0\}$ is a complex of fine sheaves. The theorem is thus established.
\end{proof}
 
\noindent
We have now some direct applications of Th. \ref{negative}.
\begin{corollary}
In the setting of Th. \ref{negative}. Then Corollaries \ref{fred1}--\ref{fred4} hold for $(M,h)$.
\end{corollary}

\begin{corollary}
Let $(M,h)$ and $(N,g)$ be as in Th. \ref{negative}. Let $V\subset \mathbb{C}\mathbb{P}^s$, $W\subset \mathbb{C}\mathbb{P}^r$ be the corresponding  Siu-Yau compactification. Assume that $V$ and $W$ are birationally equivalent. Then for each $q=0,...,m$ we have 
\begin{equation*}
 H^{m,q}_{2,\overline{\partial}}(M,h)\cong H^{m,q}_{2,\overline{\partial}}(N,g)\\
 \quad \quad \quad H^{0,q}_{2,\overline{\partial}}(M,h)\cong H^{0,q}_{2,\overline{\partial}}(N,g)
\end{equation*}
\end{corollary}

\begin{proof}
This follows from Cor. \ref{bir}.
\end{proof}


\begin{thebibliography}{99}


\bibitem{FB}
F{.} Bei,
\newblock Sobolev spaces and Bochner Laplacian on complex projective varieties and stratified pseudomanifolds.
\newblock  \emph{Journal of Geometric Analysis} 27 2017, Issue 1, pp 746--796

\bibitem{BM}
E{.} Bierstone; P{.} D{.} 
\newblock Canonical desingularization in characteristic zero by blowing up the maximum strata of a local invariant. 
\newblock {\em Invent. Math.} 128 (1997), no. 2, 207--302.

\bibitem{BLHC}
J{.} Br\"uning, M{.} Lesch, 
\newblock  Hilbert complexes.
\newblock {\em J. Funct. Anal.} 108 (1992), no. 1, 88--132. 



\bibitem{BPS}
J{.} Br\"uning, N{.} Peyerimhoff, H{.} Schr\"oder,
\newblock The $\overline{\partial}$-operator on algebraic curves.
\newblock {\em Comm. Math. Phys.}, 129 (1990), no. 3, 525--534.

\bibitem{CGM}
J{.} Cheeger, M{.} Goresky, R{.} MacPherson.
\newblock $L^2$-cohomology and intersection homology of singular algebraic varieties. 
\newblock Seminar on Differential Geometry, pp. 303--340, Ann. of Math. Stud., 102, Princeton Univ. Press, Princeton, N.J., 1982. 


%\bibitem{JPD}
%J{.}P{.} Demailly.
%\newblock Estimations $L^2$ pour l'operateur $\overline{\partial}$ d'un fibr\'e vectoriel holomorphe semi- positif au-dessus d'un vari\'et\'e K\"ahleri\'enne compl\`ete.
%\newblock {\em Ann. Sci. Ecole Norm. Sup.},  (4) 11 (1982), 457--511.

\bibitem{JPDE}
J{.}P{.} Demailly.
\newblock “L2 Hodge theory and vanishing theorems” in Introduction to Hodge Theory.
\newblock SMF/AMS Texts Monogr. 8, Amer. Math. Soc., Providence, 2002, 1--95. (2901,
2904, 2905, 2908, 2909)

\bibitem{DF}
H{.} Donnelly, C{.} Fefferman.
\newblock $L^2$-cohomology and index theorem for the Bergman metric.
\newblock {\em Ann. Math.}, (2) 118 (1983), 593--618.

\bibitem{GFI}
G{.} Fischer.
\newblock Complex analytic geometry. 
\newblock Lecture Notes in Mathematics, Vol. 538. Springer-Verlag, Berlin-New York, 1976


\bibitem{Gor}
W{.} B{.} Gordon.
\newblock An analytical criterion for the completeness of Riemannian manifolds.
\newblock {\em Proc. Amer. Math.} Soc. 37 (1973), 221--225.




\bibitem{GMM}
C{.} Grant Melles. P{.} Milman.
\newblock Classical Poincar\'e metric pulled back off singularities using a Chow-type theorem and desingularization. 
\newblock {\em Ann. Fac. Sci. Toulouse Math.} (6) 15 (2006), no. 4, 689--771.

\bibitem{GMMI}
C{.} Grant Melles. P{.} Milman.
\newblock Metrics for singular analytic spaces. 
\newblock {\em Pacific J. Math.},  168 (1995), no. 1, 61--156.

\bibitem{GRRE}
H{.} Grauert; R{.} Remmert.
\newblock Coherent analytic sheaves. Grundlehren der Mathematischen Wissenschaften [Fundamental Principles of Mathematical Sciences], 265. 
\newblock Springer-Verlag, Berlin, 1984

\bibitem{GRI}
H. Grauert, O. Riemenschneider.
\newblock Verschwindungss\"atze f\"ur analytische
Kohomologiegruppen auf komplexen R\"umen.
\newblock {\em Invent. Math.},  11 (1970), 263--292.

\bibitem{MGR}
M{.} Gromov.
\newblock K\"ahler hyperbolicity and $L^2$-Hodge theory.
\newblock {\em J. Differential Geom}. 33 (1991), no. 1, 263–292.


\bibitem{PH}
P{.} Haskell.
\newblock $L^2$-Dolbeault complexes on singular curves and surfaces. 
\newblock {\em Proc. Amer. Math. Soc.} 107 (1989), no. 2, 517--526.


\bibitem{HH}
H{.} Hironaka.
\newblock Resolution of singularities of an algebraic variety over a field of characteristic zero. I, II. 
\newblock {\em Ann. of Math.} (2) 79 (1964), 109--2.03

%\bibitem{Lesch}
%M{.} Lesch.
%\newblock Operators of Fuchs type, conical singularities, and asymptotic methods. 
%\newblock Teubner-Texte zur Mathematik [Teubner Texts in Mathematics], 136. B. G. Teubner Verlagsgesellschaft mbH, Stuttgart, 1997.



\bibitem{Mac}
R{.} MacPherson
\newblock Global questions in the topology of singular spaces. 
\newblock Proceedings of the International Congress of Mathematicians, Vol. 1, 2 (Warsaw, 1983), 213--235, PWN, Warsaw, 1984. 



\bibitem{Takeo}
T{.} Ohsawa
\newblock Hodge spectral sequence on compact K\"ahler spaces.
\newblock {\em Publ. Res. Inst. Math. Sci.},  23 (1987), 265--274.



\bibitem{P}
W{.} Pardon.
\newblock The $L^2$-$\overline{\partial}$-cohomology of an algebraic surface. 
\newblock {\em Topology},  28 (1989), no. 2, 171--195.


\bibitem{PS}
W{.} Pardon, M{.} Stern.
\newblock $L^2$-$\overline{\partial}$-cohomology of complex projective varieties.
\newblock {\em J. Amer. Math. Soc.},  4 (1991), no. 3, 603--621.



\bibitem{JRu}
J{.} Ruppenthal.
\newblock $L^2$-theory for the $\overline{\partial}$-operator on compact complex spaces.
\newblock {\em Duke Math. J.} 163 (2014), no. 15, 2887--2934.



\bibitem{LS}
L{.} Saper.
\newblock $L^2$-cohomology of K\"ahler varieties with isolated singularities.
\newblock {\em J. Differential Geom.},  36 (1992), no. 1, 89--161.

\bibitem{SiuYau}
Y{.} T{.} Siu, S.T. Yau.
\newblock  Compactification of negatively curved complete Kaehler manifolds of finite volume, Semin. 
\newblock Differential Geometry, Ann. Math. Stud. 102 (1982), 363--380


\bibitem{TaKe}
K{.} Takegoshi. 
\newblock Relative vanishing theorems in analytic spaces.
\newblock {\em Duke Math. J.} 52 (1985), no. 1, 273--279. 

\bibitem{NY}
N{.} Yeganefar.
\newblock $L^2$-cohomology of negatively curved K\"ahler manifolds of finite volume. 
\newblock {\em Geom. Funct. Anal.} 15 (2005), no. 5, 1128--1143.

\end {thebibliography}
\end{document}